\documentclass[12pt,a4paper]{amsart}

\usepackage{geometry}

\usepackage[utf8]{inputenc}
\usepackage[english]{babel}
\usepackage{amssymb}
\usepackage{amsmath}
\usepackage{amsthm}
\usepackage{amscd}
\usepackage{enumerate}
\usepackage{graphicx}
\usepackage{enumitem}

\newcommand{\vect}{\text{span}}

\DeclareMathOperator{\card}{card}

\theoremstyle{plain}
\newtheorem{theorem}{Theorem}[section]
\newtheorem{lemma}[theorem]{Lemma}
\newtheorem{corollary}[theorem]{Corollary}
\newtheorem{proposition}[theorem]{Proposition}

\theoremstyle{definition}

\theoremstyle{remark}
\newtheorem{remark}[theorem]{Remark}
\newtheorem{question}[theorem]{Question}

\numberwithin{equation}{section}

\title{Dense Lineability and algebrability of $\ell^{\infty}\setminus c_0$}
\author{Dimitris Papathanasiou}
\address{Dimitris Papathanasiou, Euler International Mathematical Institute, Pesochnaya Naberezhnaya  10,Saint Petersburg, Russia, 197376.}
\email{dpapath@bgsu.edu}

\thanks{The research is supported by ``Native towns", a social investment program of PJSC ``Gazprom Neft"}


\begin{document}

\begin{abstract}
We show that the set $\ell^{\infty}\setminus c_0$ is maximal dense-lineable and  densely strongly $\mathfrak{c}$-algebrable answering a question posed by Nestoridis and complementing a result by Garc\'ia-Pacheco, Mart\'in and Seoane-Sep\'ulveda.
\end{abstract}

\maketitle

\section{Introduction}
Let $\ell^{\infty}$ be the Banach space of real bounded sequences normed with the uniform norm, and $c_0$ be the subspace of $\ell^{\infty}$ of sequences converging to zero. Since $c_0$ is a closed subspace of $\ell^{\infty}$, the set $\ell^{\infty}\setminus c_0$ is a dense $G_{\delta}$ subset of $\ell^{\infty}$. Rosenthal \cite{Rosenthal} showed that $c_0$ is quasi-complemented in $\ell^{\infty}$ from which it immediately follows that the set $(\ell^{\infty}\setminus c_0)\cup \{0\}$ contains a closed and infinite dimensional subspace. In the same direction, Garc\'ia-Pacheco, Mart\'in and Seoane-Sep\'ulveda \cite{Seoane} showed that if we consider $\ell^{\infty}$ as a Banach algebra, endowed with the coordinatewise product, then the set $(\ell^{\infty}\setminus c_0)\cup \{0\}$ even contains a closed  infinitely generated subalgebra. 

In the recent paper \cite{Nestoridis}, Nestoridis proved that if $0<p<q<\infty$, the set $(\ell^q\setminus \ell^p)\cup \{0\}$ contains a dense linear subspace and posed the following question:
\begin{question}[Nestoridis]
Does the set $(\ell^{\infty}\setminus c_0)\cup \{0\}$ contain a dense linear subspace?
\end{question}
Nestoridis also pointed out that the fact that $\ell^{\infty}$ is non-separable could add difficulties in trying to answer the above mentioned question. Indeed, the algebraic dimension of a dense linear subspace of $\ell^{\infty}$ must be $\mathfrak{c}$ (see Proposition \ref{maximality}) which makes it difficult to control that all its non-zero elements lie outside $c_0$. On the other hand, one could argue that the fact that $\ell^{\infty}$ is non-separable, in contrast to $c_0$ provides enough space in the set  $(\ell^{\infty}\setminus c_0)\cup \{0\}$ to find a dense linear subspace.

We will provide a positive answer to Nestoridis' question by proving the following more detailed result. Borrowing the terminology from Aron et al. \cite{Aron} we will call a subset $M$ of a topological vector space $X$ \textit{maximal dense-lineable} provided $M\cup \{0\}$ contains a dense linear subspace of $X$ of dimension $\dim(X)$.

\begin{theorem}\label{basic}
The set $\ell^{\infty}\setminus c_0$ is maximal dense-lineable.
\end{theorem}

Motivated by the result of Garc\'ia-Pacheco, Mart\'in and Seoane-Sep\'ulveda, we also considered the question of whether the set $(\ell^{\infty}\setminus c_0)\cup \{0\}$ contains a dense subalgebra, where as before, we endow $\ell^{\infty}$ with the coordinatewise product. Following  Bartoszewicz and G\l\c{a}b \cite{BaGl} we will call a subset $M$ of a Banach algebra $X$ \textit{densely strongly $\mathfrak{c}$-algebrable} provided $M\cup \{0\}$ contains a dense subalgebra of $X$ generated by an algebraically independent set of cardinality $\mathfrak{c}$.

\begin{theorem}\label{algebrability}
The set $\ell^{\infty}\setminus c_0$ is densely strongly $\mathfrak{c}$-algebrable.
\end{theorem}
 
Theorems \ref{basic} and \ref{algebrability} immediately yield the following corollary which fills a gap in the literature of lineability and algebrability.

\begin{corollary}
If $0<p<\infty$ then the following hold.
\begin{enumerate}[label=(\roman*)]
\item The set $(\ell^{\infty}\setminus \ell^p)\cup \{0\}$ is maximal dense-lineable.
\item If we endow $\ell^{\infty}$ with the coordinatewise product, the set $(\ell^{\infty}\setminus \ell^p)\cup \{0\}$ is  densely strongly $\mathfrak{c}$-algebrable.
\end{enumerate}
\end{corollary}

The idea behind the proofs of Theorems \ref{basic} and \ref{algebrability} is common. For $A\subset \mathbb{N}$ we denote by $\chi_A:\mathbb{N}\rightarrow \mathbb{R}$ the characteristic sequence of $A$. The set $\{ \chi_A: A\subset \mathbb{N\}}$ generates a dense linear subspace of $\ell^{\infty}$. We want to conveniently perturb those characteristic sequences in such a way that the subspace or the subalgebra generated by the modified sequences intersects $c_0$ trivially. 

Theorem \ref{algebrability} is clearly stronger than Theorem \ref{basic} however, we provide two separate proofs for those results since we feel that the proof of Theorem \ref{basic} is more intuitive than the proof of Theorem \ref{algebrability}. In this way, one does not have to go through the more technical proof of Theorem \ref{algebrability} if interested in mere dense-lineability.

We also observe that due to the following simple fact, in our setting the maximality part of Theorem \ref{basic} and the cardinality part of Theorem \ref{algebrability} follow for free. 

\begin{proposition}\label{maximality}
\begin{enumerate}
    \item[(i)] If $E$ is a dense subspace of $\ell^{\infty}$ then $\dim(E)=\mathfrak{c}$.
    \item[(ii)] If $A$ is a dense subalgebra of $\ell^{\infty}$ (endowed with the coordinatewise product) then each set of generators of $A$ has cardinality $\mathfrak{c}$.
\end{enumerate}
\end{proposition}

\begin{proof}

    To prove (i), we observe that since $\overline{E}=\ell^{\infty}$, $E$ has to be infinite-dimensional. Now, considering linear combinations with rational coefficients of a Hamel basis of $E$, and taking into account that any dense subset of $\ell^{\infty}$ has cardinality $\mathfrak{c}$, it follows that $\dim(E)=\mathfrak{c}$.
    
    Concerning (ii), let $G$ be any set generating $A$. That means that if $\tilde{G}$ is the set of all elements of the form
    $$
    g_{i_1}^{n_1}\dots g_{i_k}^{n_k}
    $$
    where $k, n_1, \dots, n_k\in \mathbb{N}$ and $\{g_{i_1},\dots ,g_{i_k}\}$ is a finite subset of $G$, then,
    $$
    A=\vect(\tilde{G}).
    $$
    Now, $G$ has to be infinite since otherwise, $\dim(A)\leq \aleph_0$ which by (i), contradicts the fact that $A$ is dense in $\ell^{\infty}$. For the same reason we get that
    $$
    \mathfrak{c}=\card(\tilde{G})\leq \card(G) \aleph_0=\card(G) 
    $$
    from which it follows that $\card(G)=\mathfrak{c}$.
\end{proof}
\section{Proof of Theorem \ref{basic}}

We will need two lemmas, the first is a well known fact that appears in many places in literature, but we state it and prove it here for the sake of self containment.

\begin{lemma}\label{sets}
There are $\mathfrak{c}$ infinite subsets of $\mathbb{N}$ pairwise finitely intersected.
\end{lemma}

\begin{proof}
Let $(q_n)_{n=1}^{\infty}$ be an enumeration of $\mathbb{Q}\cap [0,1]$. For each $r\in [0,1]\setminus \mathbb{Q}$ we may find an increasing sequnce of natural numbers $(n_k(r))_{k=1}^{\infty}$ such that $q_{n_k(r)}\rightarrow r$, as $k\rightarrow \infty$. The sets $(n_k(r))_{k=1}^{\infty}$, for $r\in [0,1]\setminus \mathbb{Q}$ have the desired properties.
\end{proof}

The next simple fact provides our main tool in perturbing sequences with finitely many values to sequences that lie outside $c_0$ and will also be used in the proof of Theorem \ref{algebrability}. 

\begin{lemma}\label{pert_somewhere_dense}
Let $(X, \| \cdot \|)$ be a (real or complex) normed space, $a=(a_n)$ a sequence in $X$ with finite range, and $b=(b_n)$ a sequence somewhere dense in $X$. Then the sequence $a+b$ is also somewhere dense in $X$.
\end{lemma}

\begin{proof}
Let $\{c_1,\dots ,c_k\}$, for some $k\in \mathbb{N}$ be the range of $a$ and $\{N_i: i=1,\dots ,k\}$ a partition of $\mathbb{N}$ such that 
$$
a_n=c_i, \quad \text{for} \quad n\in N_i.
$$
Now, since
$$
\{b_n: n\in \mathbb{N}\}=\bigcup_{i=1}^k\{b_n: n\in N_i\}
$$
there is $j\in \{1,\dots ,k\}$ such that
$$
\{b_n: n\in N_j\}
$$
is somewhere dense. Since
$$
\{a_n+b_n: n\in N_j\}=c_j+\{b_n: n\in N_j\}
$$
the latter set is a somewhere dense subset in $X$.
\end{proof}

We may now proceed with the proof of Theorem \ref{basic}.
\begin{proof}[Proof of Theorem \ref{basic}]
If $A\subset \mathbb{N}$, we want to associate to $\chi_A$ a sequence of suitable elements from $\ell^{\infty}\setminus c_0$ converging to $\chi_A$. To this end, we let by Lemma \ref{sets} 
$$
\{(a_{(A,j)}(k))_{k=1}^{\infty}: (A,j)\in \mathcal{P}(\mathbb{N})\times \mathbb{N}\}
$$
be a family of infinite, pairwise finitely intersected subsets of $\mathbb{N}$, indexed by $\mathcal{P}(\mathbb{N})\times \mathbb{N}$ and enumerated as increasing sequences. Note that we are using the fact that the cardinality of the set $\mathcal{P}(\mathbb{N})\times \mathbb{N}$ is $\mathfrak{c}$. Let also $x=(x_n)_{n=1}^{\infty}\in \ell^{\infty}$ be a sequence which is somewhere dense in $\mathbb{R}$ (an enumeration of the rationals of $[0,1]$ would do the job). For $A\subset \mathbb{N}$ and $j\in \mathbb{N}$, we define the following elements of $\ell^{\infty}$,
$$
f_{(A,j)}=\chi_A+\frac{1}{j}\sum_{k=1}^{\infty}x_k \chi_{\{a_{(A,j)}(k)\}}.
$$
Clearly, 
$$
f_{(A,j)}\rightarrow \chi_A \quad \text{as} \quad j\rightarrow \infty
$$
hence, the set 
$$
Y=\vect \{f_{(A,j)}: A\subset \mathbb{N}, j\in \mathbb{N} \}
$$
is dense in $\ell^{\infty}$. We want to show that 
$$
Y\cap c_0=\{0\}.
$$ 
For $p\geq 1$, let $c_1,\dots ,c_p\in \mathbb{R}$ be such that $c_i\neq 0$, for some $i\in \{1,\dots ,p\}$ and $(A_1, j_1),\dots ,(A_p, j_p)$ be pairwise distinct elements of $\mathcal{P}(\mathbb{N})\times \mathbb{N}$. We consider $k_0 \in \mathbb{N}$ satisfying that
$$
\{a_{(A_i,j_i)}(k): k\geq k_0\}\cap \{a_{(A_l,j_l)}(k): k\geq k_0\}=\emptyset
$$
for all $l\neq i$, with $l \in \{1,\dots p\}$. For $k\geq k_0$, we then have that
\begin{align*}
    &c_1f_{(A_1,j_1)}(a_{(A_i,j_i)}(k))+\dots +c_pf_{(A_p,j_p)}(a_{(A_i,j_i)}(k))=\\
    &c_1\chi_{A_1}(a_{(A_i,j_i)}(k))+\dots +c_p\chi_{A_p}(a_{(A_i,j_i)}(k))+\frac{1}{j_i}c_ix_k.
\end{align*}
This, together with Lemma \ref{pert_somewhere_dense} give that the sequence
$$
(c_1f_{(A_1,j_1)}(a_{(A_i,j_i)}(k))+\dots +c_pf_{(A_p,j_p)}(a_{(A_i,j_i)}(k)))_{k=1}^{\infty}
$$
is somewhere dense in $\mathbb{R}$. In particular, 
$$
c_1f_{(A_1,j_1)}+\dots +c_pf_{(A_p,j_p)}\notin c_0
$$ 
which shows that $\ell^{\infty}\setminus c_0$ contains a dense linear subspace. An application of Proposition \ref{maximality} (i) concludes the proof. 
\end{proof}

\section{Proof of Theorem \ref{algebrability}}

The way we perturbed the characteristic sequences in the proof of Theorem \ref{basic} does not work for proving Theorem \ref{algebrability}. The problem is that the product of two sequences supported on finitely intersected subsets of $\mathbb{N}$ will be a finitely supported sequence thus it will belong to $c_0$. Obviously, we cannot hope on separating totally an uncountable family of subsets of $\mathbb{N}$. Thus, we establish the following result that allows us to overcome this obstacle. We thank the anonymous referee for indicating to us this short, topological argument.

\begin{lemma}\label{somewhere_dense_sequences}
There is a family of cardinality $\mathfrak{c}$  
$$
(f_{\gamma})_{\gamma <\mathfrak{c}}\subset [0,1]^{\mathbb{N}}
$$ 
such that for each finite subset $\{\gamma_1,\dots ,\gamma_k\}\subset \mathfrak{c}$, where $k\in \mathbb{N}$, the set
$$
\{ (f_{\gamma_1}(n),\dots ,f_{\gamma_k}(n)): n\in \mathbb{N}\}
$$
is dense in $[0,1]^k$.
\end{lemma}

\begin{proof}
The proof is an immediate consequence of the fact that the Tychonoff cube $[0,1]^{\mathfrak{c}}$ endowed with the product topology, is separable (see for instance \cite[Theorem 16.4 (c)]{Wil}). Indeed, let 
$$
\{x_n: n\in \mathbb{N}\}
$$
be a dense subset of $[0,1]^{\mathfrak{c}}$. For each $\gamma<\mathfrak{c}$, we define the element $f_{\gamma}\in [0,1]^{\mathbb{N}}$  by
$$
f_{\gamma}(n)=x_n(\gamma), \quad n\in \mathbb{N}.
$$
The fact that the family $(f_{\gamma})_{\gamma <\mathfrak{c}}$ satisfies the claim follows by the definition of the product topology.
\end{proof}

\begin{proof}[Proof of Theorem \ref{algebrability}]
By Lemma \ref{somewhere_dense_sequences} and the fact that the cardinality of $\mathcal{P}(\mathbb{N})\times \mathbb{N}$ is $\mathfrak{c}$, for $A\subset \mathbb{N}$ and $j\in \mathbb{N}$, we get sequences 
$$
f_{(A,j)}\in \ell^{\infty}
$$
with $\|f_{(A,j)}\|\leq \frac{1}{j}$ and satisfying that if 
$$
\{(A_1, j_1),\dots ,(A_k, j_k)\}
$$ 
with $k\in \mathbb{N}$, is a finite subset of $\mathcal{P}(\mathbb{N})\times \mathbb{N}$, then
$$
\{ (f_{(A_1, j_1)}(n),\dots ,f_{(A_k, j_k)}(n)): n\in \mathbb{N}\}
$$
is somewhere dense in $\mathbb{R}^k$. Now, let us set
$$
h_{(A, j)}=\chi_A+ f_{(A,j)}, \quad (A,j)\in \mathcal{P}(\mathbb{N})\times \mathbb{N}.
$$
Clearly,
$$
h_{(A, j)}\rightarrow \chi_A \quad \text{as} \quad j\rightarrow \infty
$$
hence, the set 
$$
\{ h_{(A, j)}: A\subset \mathbb{N}, j\in \mathbb{N}\}
$$ 
spans a dense subspace of $\ell^{\infty}$. We will show that if $B$
is the algebra generated by $\{ h_{(A, j)}: A\subset \mathbb{N}, j\in \mathbb{N}\}$, then 
$$
B\cap c_0=\{0\}.
$$
Let for $k\in \mathbb{N}$, $P\in \mathbb{R}[x_1,\dots x_k]$ be a non-zero polynomial vanishing at the origin. The zero set of $P$ is nowhere dense in $\mathbb{R}^k$. If $(A_1, j_1),\dots ,(A_k, j_k)$ are distinct elements of $\mathcal{P}(\mathbb{N})\times \mathbb{N}$, since the range of the sequence
$$
(\chi_{A_1}(n),\dots , \chi_{A_k}(n))_{n=1}^{\infty}
$$ 
is finite, Lemma \ref{pert_somewhere_dense} ensures that
$$
\{ (h_{(A_1, j_1)}(n),\dots , h_{(A_k, j_k)}(n)): n\in \mathbb{N}\}
$$
is somewhere dense in $\mathbb{R}^k$. Hence, there is $\delta >0$ and a subsequence $(n_l)$ of natural numbers, such that
$$
|P(h_{(A_1, j_1)}(n_l),\dots , h_{(A_k, j_k)}(n_l))|\geq \delta.
$$
In particular, since the underlying product is the coordinatewise, we get that
$$
|P(h_{(A_1, j_1)},\dots , h_{(A_k, j_k)})(n_l)|=|P(h_{(A_1, j_1)}(n_l),\dots , h_{(A_k, j_k)}(n_l))|\geq \delta
$$
which yields that
$$
P(h_{(A_1, j_1)},\dots , h_{(A_k, j_k)})\notin c_0
$$
hence, $B\cap c_0=\{0\}$ and the family $\{ h_{(A, j)}: A\subset \mathbb{N}, j\in \mathbb{N}\}$ is algebraically independent. Now, either by noticing that
$$
\card\{ h_{(A, j)}: A\subset \mathbb{N}, j\in \mathbb{N}\}=\mathfrak{c}
$$
or by an immediate application of Proposition \ref{maximality} (ii) the proof is complete. 
\end{proof}

\begin{remark}
The proofs of Theorem \ref{basic} and Theorem \ref{algebrability} can be easily modified to provide the same results in the complex case.
\end{remark}

\begin{remark}
The proofs of Theorem \ref{basic} and Theorem \ref{algebrability} actually show that the set 
$$
\ell^{\infty}\setminus c
$$ 
is maximal dense-lineable and  densely strongly $\mathfrak{c}$-algebrable, where $c$ is the space of real convergent sequences. For the analogue of Theorem \ref{algebrability} one would need to use the fact that the image under any non-constant polynomial of several variables of an open set has non-empty interior. 
\end{remark}

\textbf{Acknowledgement.} The author would like to thank the anonymous referee for his/her insightful comments which substantially improved the paper.


\begin{thebibliography}{99}
\bibitem{Aron} R.M. Aron, L. Bernal-Gonz\'alez, D. Pellegrino and J.B. Seoane-Sep\'ulveda, \emph{Lineability: The Search
for Linearity in Mathematics.} 
\newblock Monographs and Research Notes in Mathematics, Chapman and
Hall/CRC, 2015.

\bibitem{BaGl} A. Bartoszewicz and S. G\l\c{a}b
\emph{Strong algebrability of sets of sequences and functions.}
\newblock Proc. Amer. Math. Soc., 141 (2013), 827-835


\bibitem{Seoane} F. J. Garc\'ia-Pacheco, M. Mart\'in and J.B. Seoane-Sep\'ulveda,
\emph{Lineability, spaceability, and algebrability of certain subsets of function spaces.}
\newblock Taiwanese J. Math. 13 (2009) 1257--1269

\bibitem{Nestoridis} V. Nestoridis, \emph{A project about chains of spaces, regarding topological and algebraic genericity and spaceability.} 
\newblock ArXiv:2005.01023, 2020.

\bibitem{Rosenthal} H. P. Rosenthal, 
\emph{On quasi-complemented subspaces of Banach spaces.} 
\newblock Proc. Nat.
Acad. Sci. U.S.A. 59 (1968), 361–364.

\bibitem{Wil} S. Willard, \emph{General topology.} 
\newblock Mineola, N.Y. : Dover Publications, 2004.


\end{thebibliography}
\end{document}